\documentclass[11pt,twoside,leqno]{amsart}
\usepackage{mathptmx}
\usepackage{amssymb, amsmath, amsthm, enumerate}
\usepackage[cc]{titlepic}
\usepackage{graphicx}
\usepackage{amsmath, amsthm, amscd, amsfonts, amssymb, graphicx, color}
\usepackage[bookmarksnumbered, colorlinks, plainpages]{hyperref}
\usepackage[cc]{titlepic}
\usepackage{mathptmx}
\newtheorem{theorem}{Theorem}[section]
\newtheorem{lemma}[theorem]{Lemma}
\newtheorem{proposition}[theorem]{Proposition}

\theoremstyle{definition}
\newtheorem{definition}[theorem]{Definition}
\newtheorem{example}[theorem]{Example}

\theoremstyle{remark}
\newtheorem{remark}[theorem]{Remark}
\theoremstyle{proof}
\numberwithin{equation}{section}

\begin{document}
\setcounter{page}{1}
\vspace*{1.0cm}

\title[Approximate fixed point theorems of cyclical contraction mapping on $G$-metric spaces
]{Approximate fixed point theorems of cyclical contraction mapping on $G$-metric spaces } 
\author[S. A. M. Mohsenialhosseini] {  S. A. M. Mohsenialhosseini}
\date{}
\maketitle

\vskip 2mm

{\footnotesize \noindent {\bf Abstract.} In this paper, we will first introduce a new class of operators and  contraction mapping for a cyclical map $T$ on $G$-metric spaces and  the approximate fixed point property.  Also, we prove two general lemmas regarding  approximate fixed Point  of  cyclical contraction mapping on $G$-metric spaces.  Using these results we  prove several approximate fixed point theorems for a new class of operators such as Chatterjeat,  Zamfirescu, Mohseni, Mohsenialhosseini on $G$-metric spaces (not necessarily complete). 
These results can be exploited to establish new approximate fixed point  theorems for cyclical contraction maps on $G$-metric space. In addition, there is a new class of cyclical operators and contraction mapping on $G$-metric space (not necessarily complete) which do not need to be continuous. Finally, examples are given to support the usability of our results.

\vskip 1mm

\vskip 1mm

}
\renewcommand{\thefootnote}{}
\footnotetext{ 
 Vali-e-Asr University, Rafsanjan, Iran,\\
e-mail: amah@vru.ac.ir; mohsenialhosseini@gmail.com
\par
 }
\vskip 6mm
\section{\bf\large  Introduction  }
\vskip 6mm
 Fixed point theory is a very popular tool in solving existence
problems in many branches of Mathematical Analysis and its applications. In physics and engineering fixed point technique has been used in areas like image retrieval,  signal processing and the study of existence and uniqueness of solutions for a class of nonlinear integral equations.
Some recent work on fixed point theorems of integral type in $G$-metric spaces,  stability of functional difference equation can be found in \cite {Sha,Raf} and the references therein.
  
 \par  In 1968, Kannan (see \cite {Kan} ) proved a fixed point theorem for operators which need not be continuous.
Further, Chatterjea (see\cite {Cha}), in 1972, also proved a fixed point theorem for discontinuous mapping, which is actually a kind of dual of Kannan mapping.
  In 1972, by combining the above three independent  contraction conditions above, Zamfirescu (see \cite{Zam}) obtained another fixed point result for operators which satisfy the following. In 2001, Rus (see \cite{Rus}) defined $\alpha-$contraction. 
  In \cite{Ber2}, the author obtained a different   contraction condition,  also he formulated a corresponding fixed point theorem.
 In 2006, Berinde (see \cite{Ber3}) obtained some result on $\alpha-$contraction for approximate fixed point in metric space.  Miandaragh et al. \cite{Mia,Mia1} obtained some result on approximate fixed points in metric space.   
\par On the other hand, in 2006, Mustafa and Sims \cite {Mus,Mus1} introduced the notion of
generalized metric spaces or simply G-metric spaces. Many researchers have obtained   fixd point, coupled fixed point, coupled common fixed point results on  G-metric spaces (see \cite {Ayd, Cha1, Sha}).
 \par In 2011, Mohsenalhosseini et al \cite {Moh},  introduced the approximate best proximity pairs and proved the approximate best proximity pairs  property for it. Also, In 2012 , Mohsenalhosseini et al \cite {Moh1}, introduced the approximate fixed point for completely  norm space and map $T_\alpha$ and proved the approximate fixed point  property for it. In 2014 , Mohsenalhosseini   \cite {Moh2}  introduced the  Approximate best proximity pairs on  metric space for contraction maps. Also, Mohsenalhosseini in \cite {Moh3} introduced the approximate fixed point in G-metric spaces for various types of operators. Recently,  in 2017 Mohsenialhosseini \cite {Moh4}  introduced the approximate fixed points  of operators  on $G$-metric spaces.
The aim of this paper is to introduce the new classes of operators and  contraction maps (not necessarily continuous) regarding approximate fixed point and diameter approximate fixed point for cyclical contraction mapping on $G$-metric spaces. Also, we give some illustrative example of our main results. 
\section{\bf\large Preliminaries}
This section recalls the following notations and the ones that will be used in what follows.
In 2003, kirk et al \cite {Kir}, obtained
an extension of Banach’s fixed point theorem by considering a cyclical operator.
\begin {definition} \cite {Kir} \label{1.1}
{\it Let $\{X_i\}_{i=1}^m$ be  nonempty susets of a complete metric space $X.$ A mapping $T:\cup_{i=1}^{m}X_i \rightarrow \cup_{i=1}^{m}X_i$ satisfies the following condition(where $X_{i+1}=X_1)$ 
 $$T(X_1)\subseteq X_2,...,T(X_{m-1})\subseteq X_m, T(X_m)\subseteq X_1,$$  
is called a cyclical operator.}
\end{definition}
\begin {definition}\cite {Mus}  \label{1.1} Let $X$ be a nonempty set and let $G : X \times X \times X \longrightarrow R^{+}$
be a function satisfying the following properties:\\
(G1) $ G(x, y, z) = 0 ~if ~and ~only~ if~ x = y = z; $\\
(G2) $ 0 < G(x,x,y) for~ all~ x, y \in X ~with~ x \neq y;$\\
(G3) $ G(x, x, y) \leq G(x, y, z) for~ all~ x, y, z \in X with ~z \neq y;$\\
(G4) $G(x, y, z) = G(x, z, y) = G(y, z, x) =\cdots  (symmetry~  in ~all~three~variables );$\\
(G5) $G(x, y, z) \leq G(x, a, a) + G(a, y, z) for ~all~ x, y, z, a \in X (~ rectangle~  inequality).$\\
Then, the function $G$ is called generalized metric or, more specifically 
 $G-metric$ on $X,$ and the pair $(X,G)$ is called a $G-metric$
space.
\end{definition} 
\begin{proposition} \label{2.2} \cite {Mus}
Every $G$-metric $(X, G)$ defines a metric space $(X, d_G)$ by \\
1) $d_G(x, y)= G(x, y, y)+ G(y, x, x).$\\
if $(X, G)$ is a symmetric $G$- metric space. Then \\
2) $d_G(x, y)= 2G(x, y, y).$
\end{proposition}

\begin {definition} \cite {Moh1} \label{1.3}
{\it Let $T:X\rightarrow X,$ $\epsilon>0,$  $x_0\in X.$
 Then $x_0\in X$ is an $\epsilon-$fixed point for $T$ if $\|Tx_0- x_0\|<\epsilon.$}
\end{definition}

\begin{remark}  \cite {Moh1} \label{1.2}
{\it In this paper we will denote the set of all $\epsilon-$ fixed points of $T$, for a
given $\epsilon$, by : 
\begin{eqnarray*}
F_\epsilon (T)=\{x\in X~|~x ~is~ an~ \epsilon-fixed~ point~of~T\}.
\end{eqnarray*}}
\end{remark}

\begin{definition} \cite {Moh1} Let $T:X\rightarrow X.$  Then $T$ has the approximate fixed point property (a.f.p.p) if $$\forall \epsilon>0,~F_\epsilon (T)\neq\varnothing.$$
\end{definition}

\begin{lemma}  \cite {Moh5} \label{1.6}
{\it Let $\{X_i\}_{i=1}^m$ be nonempty  subsets of a    metric space $X$ and  $T:\cup_{i=1}^{m}X_i \rightarrow \cup_{i=1}^{m}X_i$ be a  cyclical operator. Let  $x_0\in \cup_{i=1}^{m}X_i$ and $\epsilon>0.$ If $T:\cup_{i=1}^{m}X_i \rightarrow \cup_{i=1}^{m}X_i$ is asymptotically regular at  each point $x_0\in \cup_{i=1}^{m}X_i$,  then $T$ has an $\epsilon-$ fixed point.}
\end{lemma}

\begin{lemma}  \cite {Moh5}\label{1.4}
 {\it Let $\{X_i\}_{i=1}^m$ be nonempty  subsets of a    metric space $X$ and  $T:\cup_{i=1}^{m}X_i \rightarrow \cup_{i=1}^{m}X_i$ be a  cyclical operator.  Let $x_0\in \cup_{i=1}^{m}X_i$ and $\epsilon>0.$ If $d(T^{n}(x_0),T^{n+k}(x_0)) \rightarrow 0$ as $n\rightarrow\infty$ for some $k>0,$  then $T^k$ has an $\epsilon-$ fixed point.}
\end{lemma}

\begin{lemma} \cite {Moh5}\label{1.5} 
{\it Let $\{X_i\}_{i=1}^m$ be nonempty subsets of a metric space $X,$ $T:\cup_{i=1}^{m}X_i \rightarrow \cup_{i=1}^{m}X_i$ a  cyclical operator and $\epsilon >0.$
We assume that:
\item [\rm{(i)}]~  $F_\epsilon (T)\neq\emptyset;$
\item [\rm{(ii)}]~  $\forall \theta>0,~ \exists \phi(\theta)>0~such~that;$
\begin{eqnarray*}
d(x,y)-d(Tx,Ty) &\leq& \theta \Rightarrow d(x,y)\leq \phi(\theta),~\forall x,y\in F_\epsilon (T)\neq\emptyset. 
\end{eqnarray*}
Then: $$\delta(F_\epsilon (T))\leq \phi(2\epsilon).$$}
\end{lemma}

\section{Main result}
We begin with two lemmas which will be used in order to prove all the results given in third section. Let $(X,G)$ be a  $G-metric$ space.
\begin {definition}
Let $\{X_i\}_{i=1}^m$ be  nonempty susets of a
  $G-metric$ space $X$ and $T:\cup_{i=1}^{m}X_i \rightarrow \cup_{i=1}^{m}X_i$ be a  cyclical operator. Let  $\epsilon>0$ and  $x_0 \in \cup_{i=1}^{m}X_i .$ Then  $x_0$ is an $\epsilon$- fixed point of $T$ if 
$$[G(x_0, Tx_0, Tx_0)+G(Tx_0, x_0, x_0)]< \epsilon.$$
\end{definition}

\begin{remark} In this paper we will denote the set of all $\epsilon$-fixed points of $T$, for a given $\epsilon$, by: 
$$F^{\epsilon}_G (T)=\{x\in \cup_{i=1}^{m}X_i\mid x~ is~ an~ \epsilon-fixed~ point~ of~T \}.$$
\end{remark}

\begin {definition} 
  Let $\{X_i\}_{i=1}^m$ be  nonempty closed subsets of a  $G-metric$ space $X,$ $T:\cup_{i=1}^{m}X_i \rightarrow \cup_{i=1}^{m}X_i$ be a  cyclical operator and $\epsilon >0$. We define diameter of the set $F^{ \epsilon}_G (T),$ i.e.,
\begin{equation*}
\delta(F^{ \epsilon}_G (T))=\sup \{G(x,y,z):~~x,y,z\in F^{\epsilon}_G (T)\}.
\end{equation*}
\end{definition}

\begin {definition} 
Let $\{X_i\}_{i=1}^m$ are closed subsets of a  $G-metric$ space $X$ and  $T:\cup_{i=1}^{m}X_i \rightarrow \cup_{i=1}^{m}X_i$ be a  cyclical operator. Then $T$ has the approximate fixed point property (a.f.p.p) if $\forall \epsilon>0$, $$F^{\epsilon}_G  (T)\neq \emptyset.$$
\end{definition}

\begin {definition}
Let $\{X_i\}_{i=1}^m$ are closed subsets of a  $G-metric$ space $X.$ A cyclical operator  $T:\cup_{i=1}^{m}X_i \rightarrow \cup_{i=1}^{m}X_i$ is
said to be asymptotically regular at a point $x \in \cup_{i=1}^{m}X_i,$ if 
\begin{equation*}
\lim _{n \rightarrow \infty} \{G(T^{n}x, T^{n+1}x, T^{n+1}x)+ G(T^{n+1}x, T^{n}x, T^{n}x) \} =0,
\end{equation*}
 where $T^n$ denotes the $n$th iterate of $T$ at $x.$
\end{definition}

\begin{lemma}  \label{2.5}
Let $\{X_i\}_{i=1}^m$ are closed subsets of a  $G-metric$ space $X.$  If  $T:\cup_{i=1}^{m}X_i\rightarrow \cup_{i=1}^{m}X_i$ is asymptotically regular at a point $x \in \cup_{i=1}^{m}X_i,$ Then $T$ has an approximate fixed point.
\end{lemma}
\begin{proof}
The proof of Lemma is the same as the proof of Lemma \ref{1.6} for $x\in \cup_{i=1}^{m}X_i.$
\end{proof}
 
\begin{lemma} 
 {\it Let $\{X_i\}_{i=1}^m$ be nonempty  subsets of a    $G-metric$ space $X$ and  $T:\cup_{i=1}^{m}X_i \rightarrow \cup_{i=1}^{m}X_i$ be a  cyclical operator.  Let $x_0\in \cup_{i=1}^{m}X_i$ and $\epsilon>0.$ If
\begin{eqnarray*}
G(T^{n}(x_0),T^{n+k}(x_0),T^{n+k}(x_0)) +G(T^{n+k}(x_0),T^{n}(x_0),T^{n}(x_0))\rightarrow 0
\end{eqnarray*}
 as $n\rightarrow\infty$ for some $k>0,$  then $T^k$ has an $\epsilon-$ fixed point.}
\end{lemma} 
\begin{proof}
The proof of Lemma is the same as the proof of Lemma \ref{1.4} for $x\in \cup_{i=1}^{m}X_i.$
\end{proof}
 
\begin{lemma}  \label{2.8}
Let $\{X_i\}_{i=1}^m$ are closed subsets of a  $G-metric$ space $X,$ $T:\cup_{i=1}^{m}X_i\rightarrow \cup_{i=1}^{m}X_i$ a cyclical operator and $\epsilon >0$. We assume that: \\
a) $F^{ \epsilon}_G (T) \neq \emptyset;$ \\
b) $\forall \xi >0 ~\exists \psi( \xi)>0$ such that 
$$[G(x, y, y)+G(y, x, x)] - [G(Tx, Ty, Ty)+G(Ty, Tx, Tx)]< \xi\Longrightarrow$$
 $$ G(x, y, y)+G(y, x, x)\leq \psi( \xi),\quad \forall x, y \in F^{ \epsilon}_G (T).$$ Then: 
$$\delta (F^{ \epsilon}_G (T))\leq \psi (2\epsilon).$$
\end{lemma}

\begin{proof}
 The proof of Lemma is the same as the proof of Lemma \ref{1.5} for $x\in \cup_{i=1}^{m}X_i.$
 \end{proof}
\section{ Approximate fixed point for several  operator on $G-$ metric spaces}
In this section a series of  qualitative and quantitative  results will be  obtained regarding the properties of approximate fixed point. Also, by using Proposition \ref{2.2} and lemma \ref{1.4} we  prove approximate fixed point theorems and diameter approximate theorems for a new class of cyclical operators on $G$-metric spaces.
\begin{definition} \cite{Moh3}
Let $(X, d)$ be a metric space. A mapping $T:X\rightarrow  X$ is a \textbf{Mohseni operator} if there exists $\alpha \in (0, \dfrac{1}{2})$ such that  
$$d(Tx, Ty)\leq \alpha [d(x, y)+  d(Tx, Ty)].$$
\end{definition}
\begin{definition}\cite {Moh5}
Let $\{X_i\}_{i=1}^m$ be nonempty subsets of a metric space $X,$ $T:\cup_{i=1}^{m}X_i \rightarrow \cup_{i=1}^{m}X_i$ is a  $\alpha-$cyclical contraction if there exists $\alpha \in (0, \dfrac{1}{2})$ such that 
 $$d(Tx, Ty)\leq \alpha d(x, y)\quad \forall x\in X_i,y\in X_{i+1}.$$
\end{definition}

\begin{definition} \label{3.2}
Let $\{X_i\}_{i=1}^m$ be nonempty subsets of a $G-metric$ space $X,$ $T:\cup_{i=1}^{m}X_i \rightarrow \cup_{i=1}^{m}X_i$ is a  $G-\alpha-$cyclical contraction if there exists $\alpha \in (0, 1)$ such that 
\begin{align*}
 G(Tx, Ty, Ty)+ G(Ty, Tx, Tx)]&\leq 
 \alpha [G(x, y, y)+  G(y, x, x) \\
&+ G(Tx, Ty, Ty)+ G(Ty, Tx, Tx)] \quad\forall x\in X_i,y\in X_{i+1}.
\end{align*}
\end{definition}

\begin{theorem}   \label{3.3} 
{\it Let $\{X_i\}_{i=1}^m$ be nonempty subsets of a  $G-metric$ metric space $X$ and 
 Suppose $T:\cup_{i=1}^{m}X_i \rightarrow \cup_{i=1}^{m}X_i$ is a  $G-\alpha-$cyclical contraction. 
Then $T$ has an $\epsilon-$fixed point.}
 \end{theorem}
{\bf Proof:} 
 Let $\epsilon>0$ and $x\in \cup_{i=1}^{m}X_i .$ 
 \begin{eqnarray*}
G(T^{n}x,T^{n+1}x,T^{n+1}x) +G(T^{n+1}x,T^{n}x,T^{n}x)&=& 
G(T(T^{n-1}x),T(T^{n}x),T(T^{n}x)) \\
&+& G(T(T^{n}x),T(T^{n-1}x),T(T^{n-1}x))\\
&\leq &\alpha G(T^{n-1}x,T^{n}x,T^{n}x) +G(T^{n}x,T^{n-1}x,T^{n-1}x)\\
&\vdots & \\
& \leq & (\alpha)^n G(x,Tx,Tx) +G(Tx,x,x).
\end{eqnarray*}
But $\alpha \in (0,\dfrac{1}{2})$. Hence
\begin{equation*}
lim_{n\rightarrow \infty} (G(T^{n}x,T^{n+1}x,T^{n+1}x) +G(T^{n+1}x,T^{n}x,T^{n}x))=0, \forall x\in \cup_{i=1}^{m}X_i.
\end{equation*}
Hence by  Lemma \ref{2.5} it follows that $F^{\epsilon}_G (T)\neq\emptyset, \forall \epsilon>0.~\blacksquare$//

In 1972, Chatterjea (see \cite {Cha}) considered another  operator in which continuity is not imposed. Now, the  appoximate fixed point theorems by using cyclical operators  on $G$-metric spaces are obtained. 

\begin{definition} 
Let $\{X_i\}_{i=1}^m$ be nonempty subsets of a $G-metric$ metric space $X,$ $T:\cup_{i=1}^{m}X_i \rightarrow \cup_{i=1}^{m}X_i$ is a  G-Chatterjea cyclical operator if there exists $\alpha\in(0,\frac{1}{2})$ such that 
\begin{align*}
 [G(Tx, Ty, Ty)+ G(Ty, Tx, Tx)]&\leq 
 \alpha [G(x, Ty, Ty)+  G(Ty, x, x) \\
&+ G(y, Tx, Tx)+ G(Tx, y, y)] \quad\forall x\in X_i,y\in X_{i+1}.
\end{align*}
\end{definition}

\noindent {\bf Theorem 2.4.} 
{\it Let $\{X_i\}_{i=1}^m$ be nonempty subsets of a $G-metric$ space $X.$  Suppose that the mapping  $T:\cup_{i=1}^{m}X_i \rightarrow \cup_{i=1}^{m}X_i$ is a  G-Chatterjea cyclical operator. 
Then $T$ has an $\epsilon-$fixed point.}

{\bf Proof:} Let $\epsilon>0$ and $x\in \cup_{i=1}^{m}X_i.$ 
\begin{eqnarray*}
[G(T^{n}x,T^{n+1}x,T^{n+1}x) +G(T^{n+1}x,T^{n}x,T^{n}x)]&=& 
[G(T(T^{n-1}x),T(T^{n}x),T(T^{n}x)) \\
&+& G(T(T^{n}x),T(T^{n-1}x),T(T^{n-1}x))]\\
&\leq &\alpha [G(T^{n-1}x,T(T^{n}x),T(T^{n}x))+G(T(T^{n}x),T^{n-1}x,T^{n-1}x)\\
&+& G(T^{n}x,T(T^{n-1}x),T(T^{n-1}x)) +G(T(T^{n-1}x),T^{n}x,T^{n}x)]\\
&= &\alpha (G(T^{n-1}x,T^{n+1}x,T^{n+1}x)+G(T^{n+1}x,T^{n-1}x,T^{n-1}x)).\\
\end{eqnarray*}
On the other hand
\begin{eqnarray*}
G(T^{n-1}x,T^{n+1}x,T^{n+1}x)+G(T^{n+1}x,T^{n-1}x,T^{n-1}x)\\
&\leq & [G(T^{n-1}x,T^{n}x,T^{n}x)+ G(T^{n}x,T^{n-1}x,T^{n-1}x) \\
&+&G(T^{n}x,T^{n+1}x,T^{n+1}x)+ G(T^{n+1}x,T^{n}x,T^{n}x)].
\end{eqnarray*}
Then 
\begin{eqnarray*}
(1-\alpha)(G(T^{n}x,T^{n+1}x,T^{n+1}x)+ G(T^{n+1}x,T^{n}x,T^{n}x))
&\leq & \alpha(G(T^{n-1}x,T^{n}x,T^{n}x)+ G(T^{n}x,T^{n-1}x,T^{n-1}x)), 
\end{eqnarray*}
 hence
\begin{eqnarray*}
(G(T^{n}x,T^{n+1}x,T^{n+1}x)+ G(T^{n+1}x,T^{n}x,T^{n}x))&\leq &\frac{\alpha}{1-\alpha}(G(T^{n-1}x,T^{n}x,T^{n}x)+ G(T^{n}x,T^{n-1}x,T^{n-1}x))\\
&&\vdots\\
& \leq &(\frac{\alpha}{1-\alpha})^n(G(x,Tx,Tx)+G(Tx,x,x)).
\end{eqnarray*}
But $\alpha\in(0,\frac{1}{2})$ hence $\frac{\alpha}{1-\alpha}\in (0,1).$ Therfore
\begin{equation*}
lim_{n\rightarrow \infty} (G(T^{n}x,T^{n+1}x,T^{n+1}x) +G(T^{n+1}x,T^{n}x,T^{n}x))=0, \forall x\in \cup_{i=1}^{m}X_i.
\end{equation*}
Hence by  Lemma \ref{2.5} it follows that $F^{\epsilon}_G (T)\neq\emptyset, \forall \epsilon>0.~\blacksquare$
\begin{definition} \label{3.2}
Let $\{X_i\}_{i=1}^m$ be nonempty subsets of a $G-metric$ space $X,$ $T:\cup_{i=1}^{m}X_i \rightarrow \cup_{i=1}^{m}X_i$ is a  \textbf{G-Mohseni cyclical operator} if there exists $\alpha \in (0, \dfrac{1}{2})$ such that 

\begin{align*}
 [G(Tx, Ty, Ty)+ G(Ty, Tx, Tx)]&\leq 
 \alpha [G(x, y, y)+  G(y, x, x) \\
&+ G(Tx, Ty, Ty)+ G(Ty, Tx, Tx)] \quad\forall x\in X_i,y\in X_{i+1}.
\end{align*}
\end{definition}

\begin{theorem}   \label{3.3} 
{\it Let $\{X_i\}_{i=1}^m$ be nonempty subsets of a $G-metric$ space $X$ and 
 Suppose $T:\cup_{i=1}^{m}X_i \rightarrow \cup_{i=1}^{m}X_i$ is a  G-Mohseni cyclical operator. 
Then $T$ has an $\epsilon-$fixed point.}
 \end{theorem}
{\bf Proof:} 
 Let $\epsilon>0$ and $x\in \cup_{i=1}^{m}X_i .$ 
 \begin{eqnarray*}
[G(T^{n}x,T^{n+1}x,T^{n+1}x) +G(T^{n+1}x,T^{n}x,T^{n}x)]&=& 
[G(T(T^{n-1}x),T(T^{n}x),T(T^{n}x)) \\
&+& G(T(T^{n}x),T(T^{n-1}x),T(T^{n-1}x))]\\
&\leq &\alpha [G(T^{n-1}x,T^{n}x,T^{n}x)+G(T^{n}x,T^{n-1}x,T^{n-1}x)\\
&+& G(T^{n}x,T^{n+1}x,T^{n+1}x) +G(T^{n+1}x,T^{n}x,T^{n}x)].
\end{eqnarray*}
Therefore, 
\begin{equation*}
(1-\alpha)[G(T^{n}x,T^{n+1}x,T^{n+1}x) +G(T^{n+1}x,T^{n}x,T^{n}x)]\leq \alpha G(T^{n-1}x,T^{n}x,T^{n}x)+G(T^{n}x,T^{n-1}x,T^{n-1}x).
\end{equation*}
So,
\begin{eqnarray*}
[G(T^{n}x,T^{n+1}x,T^{n+1}x) +G(T^{n+1}x,T^{n}x,T^{n}x)]&\leq &  \frac{\alpha}{(1-\alpha)} [G(T^{n-1}x,T^{n}x,T^{n}x)+G(T^{n}x,T^{n-1}x,T^{n-1}x)]\\
&&\vdots \\
& \leq & (\frac{\alpha}{(1-\alpha)})^n(G(x,Tx,Tx)+G(Tx,x,x)).
\end{eqnarray*}
But $\alpha \in (0,\dfrac{1}{2})$, therefore $( \dfrac{\alpha}{1-\alpha}) \in (0, 1).$ Hence
\begin{equation*}
lim_{n\rightarrow \infty} (G(T^{n}x,T^{n+1}x,T^{n+1}x) +G(T^{n+1}x,T^{n}x,T^{n}x))=0, \forall x\in \cup_{i=1}^{m}X_i.
\end{equation*}
Hence by  Lemma \ref{2.5} it follows that $F^{\epsilon}_G (T)\neq\emptyset, \forall \epsilon>0.~\blacksquare$ 
\begin{example} Consider the sets:
$A_1=\{\frac{1}{k}\}_{k=1}^{\infty}\cup\{\frac{-1}{2k}\}_{k=1}^{\infty}$ and
$A_2=\{\frac{-1}{k}\}_{k=1}^{\infty}\cup\{\frac{1}{2k-1}\}_{k=1}^{\infty}.$ 
Define the map  $T:\cup_{i=1}^{2}A_i\rightarrow \cup_{i=1}^{2}A_i$ as
\[Tx=\left\{\begin {array}{cl}
\frac{-x}{x+4}\quad if \quad x\in A_1\\
\\
\frac{-x}{4}\quad if \quad x\in A_2
 \end{array}\right.\]
It is easily to be checked that $T(A_1) \subseteq A_2$  and $T(A_2) \subseteq A_1.$
For any $x \in A_1$ and $y\in A_2$ and Proposition \ref{2.2} we have the chain of inequalities 
\begin{eqnarray*}
[G(Tx, Ty, Ty)+ G(Ty, Tx, Tx)]=d_G(Tx,Ty)&= & |\frac{x}{x+4}-\frac{y}{4}|\\
&\leq &\frac{1}{3}(|x|+|y|)\\
&\leq &\frac{1}{3}(|x-y|+|\frac{x}{x+4}-\frac{y}{4}|)\\
&= &\frac{1}{3}(d_G(x,y)+d_G(Tx,Ty))\\
&=&\frac{1}{3} [G(x, y, y)+  G(y, x, x) + G(Tx, Ty, Ty)+ G(Ty, Tx, Tx)].
\end{eqnarray*}
So $T$ satisfies all the conditions of Theorem \ref{3.3} and thus it has a approximate fixed point. 
\end{example}
\begin{example} \label{3.8}
 Let $X$ be a subset in $R$ endowed with the usual metric. Suppose $A_1=]0,0.8]$ and $A_2=]0,\frac{1}{2}].$ 
 Define the map $T:\cup_{i=1}^{2}A_i\rightarrow \cup_{i=1}^{2}A_i$ as
$Tx=\frac{x}{4}$ for all $x\in\cup_{i=1}^{2}A_i.$
It is easily to be checked that  $T(A_1) \subseteq A_2$  and $T(A_2) \subseteq A_1.$
For any $x \in A_1$ and $y\in A_2$ and Proposition \ref{2.2} we have the chain of inequalities
\begin{eqnarray*}
[G(Tx, Ty, Ty)+ G(Ty, Tx, Tx)]=d_G(Tx,Ty)&= & |\frac{x}{4}-\frac{y}{4}|\\
&\leq &\frac{1}{3}(|x-y|+|\frac{x}{4}-\frac{y}{4}|)\\
&= &\frac{1}{3}(d_G(x,y)+d_G(Tx,Ty))\\
&=&\frac{1}{3} [G(x, y, y)+  G(y, x, x) + G(Tx, Ty, Ty)+ G(Ty, Tx, Tx)].
\end{eqnarray*}
\end{example}
So $T$ satisfies all the conditions of Theorem \ref{3.3} and thus for every $\epsilon>0,$ $ F_{ \epsilon} (T))\neq\emptyset.$ on the other hand take $0<\epsilon<\frac{1}{2}$ and select $x_0\in\cup_{i=1}^{2}A_i$ such that $x_0<\frac{4}{3}\epsilon.$ Then 
\begin{equation*}
d(Tx ,x) =|\frac{x}{4}-x|\leq \epsilon.
\end{equation*}
Hence by by Proposition \ref{2.2}, $G(Tx,x,x)+G(x,Tx,Tx)\leq \epsilon$.
So $T$ has an approximate fixed point which implies that $ F_{ \epsilon} (T))\neq\emptyset.$ On the contrary, there is no fixed point of $T$ in $\cup_{i=1}^{2}A_i.$

\par By combining the three independent contraction conditions: $G-\alpha-$cyclical contraction, G-Mohseni cyclical, and G-Chatterjea cyclical operators we obtain another  approximate fixed point result for operators which satisfy the following.

\begin{definition} \label{3.9}
Let $\{X_i\}_{i=1}^m$ be nonempty subsets of a $G-metric$ space $X,$  $T:\cup_{i=1}^{m}X_i \rightarrow \cup_{i=1}^{m}X_i$ is a  \textbf{G-Mohsenialhosseini cyclical operator} if there exists $\alpha,\beta,\gamma \in R,$ 
$\alpha\in[0,1[,\beta\in [0,\frac{1}{2}[,\gamma\in [0,\frac{1}{2}[$ such that for all $x \in _{i=1}^{m}X_i ,y\in X_{i+1}$
 at least one of the following is true:
 \begin{equation*}
(i) G(Tx, Ty, Ty)+ G(Ty, Tx, Tx)]\leq 
 \alpha [G(x, y, y)+  G(y, x, x) + G(Tx, Ty, Ty)+ G(Ty, Tx, Tx)];
\end{equation*}
\begin{equation*}
(ii) [G(Tx, Ty, Ty)+ G(Ty, Tx, Tx)]\leq 
 \beta  [G(x, y, y)+  G(y, x, x) + G(Tx, Ty, Ty)+ G(Ty, Tx, Tx)];
\end{equation*}
 \begin{equation*}
(iii)  [G(Tx, Ty, Ty)+ G(Ty, Tx, Tx)]\leq 
\gamma [G(x, Ty, Ty)+  G(Ty, x, x) + G(y, Tx, Tx)+ G(Tx, y, y)].
\end{equation*}
\end{definition}
\begin{theorem} \label{3.10}
{\it Let $\{X_i\}_{i=1}^m$ be nonempty subsets of a $G-metric$ metric space $X$ and 
 Suppose $T:\cup_{i=1}^{m}X_i \rightarrow \cup_{i=1}^{m}X_i$ is a  G-Mohsenialhosseini cyclical operator. 
Then $T$ has an $\epsilon-$fixed point.}
\end{theorem} 
{\bf Proof:} Let $x,y \in \cup_{i=1}^{m}X_i.$ Supposing $ii$) holds, we have that:
\begin{eqnarray*}
[G(Tx, Ty, Ty)+ G(Ty, Tx, Tx)]&\leq & 
 \beta  [G(x, y, y)+  G(y, x, x) + G(Tx, Ty, Ty)+ G(Ty, Tx, Tx)];\\
&\leq &\beta [G(x,Tx, Tx)+  G(Tx, x, x) + G(Tx, y, y)+ G(y, Tx, Tx)\\
 &+&G(Tx, Ty, Ty)+ G(Ty, Tx, Tx)]\\
&\leq &\beta [G(x,Tx, Tx)+  G(Tx, x, x)+G(Tx, x, x)+ G(x, Tx, Tx)\\
&+& G(x, y, y)+  G(y, x, x)+G(Tx, Ty, Ty)+ G(Ty, Tx, Tx)]\\
&= &2\beta [G(x,Tx, Tx)+  G(Tx, x, x)]+\beta [G(x, y, y)\\
&+&  G(y, x, x)]+\beta [G(Tx, Ty, Ty)+ G(Ty, Tx, Tx)].
\end{eqnarray*}
Thus
\begin{equation}
[G(Tx, Ty, Ty)+ G(Ty, Tx, Tx)]\leq \frac{2\beta}{1-\beta}[G(x,Tx, Tx)+ G(Tx, x, x)]+\frac{\beta}
{1-\beta}[G(x, y, y)+  G(y, x, x)].
\end{equation}

Supposing $iii$
) holds, we have that:
\begin{eqnarray*}
[G(Tx, Ty, Ty)+ G(Ty, Tx, Tx)]&\leq& \gamma[G(x,Ty, Ty)+  G(Ty, x, x) +[G(y,Tx, Tx)+  G(Tx, y, y) ]\\
&\leq &\gamma [G(x, y, y)+  G(y, x, x)+G(y, Ty, Ty)+ G(Ty, y, y)]\\
&+&\gamma [G(y, Ty, Ty)+ G(Ty, y, y)+G(Ty, Tx, Tx)+ G(Tx, Ty, Ty)]\\
&=&\gamma [G(Tx, Ty, Ty)+ G(Ty, Tx, Tx)]+2\gamma [G(y, Ty, Ty)\\
&+& G(Ty, y, y)]+\gamma [G(x, y, y)+  G(y, x, x)].
\end{eqnarray*}
Thus
\begin{equation}
[G(Tx, Ty, Ty)+ G(Ty, Tx, Tx)]\leq\frac{2\gamma}{1-\gamma}[G(y, Ty, Ty)+ G(Ty, y, y)]+\frac{\gamma}{1-\gamma}[G(x, y, y)+  G(y, x, x)].
\end{equation}

Similarly:
\begin{eqnarray*}
[G(Tx, Ty, Ty)+ G(Ty, Tx, Tx)]&\leq& \gamma[G(x, Ty, Ty)+ G(Ty, x, x)+G(y, Tx, Tx)+ G(Tx, y, y)]\\
&\leq &\gamma [G(x, Tx, Tx)+ G(Tx, x, x)+G(Tx, Ty, Ty)+ G(Ty, Tx, Tx)]\\
&+&\gamma [[G(y, x, x)+  G(x, y, y)+G(x, Tx, Tx)+ G(Tx, x, x)]\\
&=&\gamma [G(Tx, Ty, Ty)+ G(Ty, Tx, Tx)]\\
&+&2\gamma [G(x, Tx, Tx)+ G(Tx, x, x)]+\gamma [G(x, y, y)+  G(y, x, x)].
\end{eqnarray*}
Then
\begin{equation}
[G(Tx, Ty, Ty)+ G(Ty, Tx, Tx)]\leq\frac{2\gamma}{1-\gamma}[G(x, Tx, Tx)+ G(Tx, x, x)]+\frac{\gamma}{1-\gamma} [G(x, y, y)+  G(y, x, x)].
\end{equation}
Therefore for $T$ satisfying at least one of the conditions (i), (ii), (iii) we have that
\begin{equation}
[G(Tx, Ty, Ty)+ G(Ty, Tx, Tx)]\leq 2\eta [G(x, Tx, Tx)+ G(Tx, x, x)]+\eta [G(x, y, y)+  G(y, x, x)],
\end{equation}
and
\begin{equation}
[G(Tx, Ty, Ty)+ G(Ty, Tx, Tx)]\leq 2\eta [G(y, Ty, Ty)+ G(Ty, y, y)]+\eta [G(x, y, y)+  G(y, x, x)],
\end{equation}
where  
$\eta:=max\{\alpha,\frac{\beta}{1-\beta},\frac{\gamma}{1-\gamma}\},$ hold. Using these conditions implied by (i) - (iii) and taking $x\in \cup_{i=1}^{m}X_i,$ we have:
\begin{eqnarray*}
[G(T^{n}x,T^{n+1}x,T^{n+1}x) +G(T^{n+1}x,T^{n}x,T^{n}x)]&=& 
[G(T(T^{n-1}x),T(T^{n}x),T(T^{n}x)) \\
&+& G(T(T^{n}x),T(T^{n-1}x),T(T^{n-1}x))]\\
&\leq^{(4.4)}&2\eta [G(T^{n-1}x,T(T^{n-1})x,T(T^{n-1})x)\\
&+&G(T(T^{n-1})x,T^{n-1}x,T^{n-1}x)]\\
&+& \eta[G(T^{n-1}x,T^{n}x,T^{n}x) +G(T^{n}x,T^{n-1}x,T^{n-1}x)]\\
&=& 3\eta [G(T^{n-1}x,T^{n}x,T^{n}x) +G(T^{n}x,T^{n-1}x,T^{n-1}x)]\\
&&\vdots \\
& \leq &(3\eta)^n(G(x,Tx,Tx)+G(Tx,x,x)).
\end{eqnarray*}
Therfore
\begin{equation*}
Lim_{n\rightarrow\infty}[G(T^{n}x,T^{n+1}x,T^{n+1}x) +G(T^{n+1}x,T^{n}x,T^{n}x)] =0, ~\forall x\in \cup_{i=1}^{m}X_i.
\end{equation*}
Now by Lemma \ref{2.5} it follows that $F^{\epsilon}_G (T)\neq\emptyset, \forall \epsilon>0.$ $~\blacksquare$
\begin{example} \label{4.12} 
 Let $X=[0,\infty)$ and let $d$ be usual metric on $X.$ Suppose $A_1=[0.1,2]$ and $A_2=[0.1,1].$ 
 Define the map $T:\cup_{i=1}^{2}A_i\rightarrow \cup_{i=1}^{2}A_i$ as
\[Tx=\left\{\begin {array}{cl}
0\quad \quad  x\in [0,1-\beta)\\
\frac{x}{4}\qquad   x\in [1-\beta,1)\\
\frac{1-\beta}{4}\qquad \quad x\in [1,2]
 \end{array}\right.\]
It is easily to be checked that  $T(A_1) \subseteq A_2$  and $T(A_2) \subseteq A_1.$
For any $x ,y\in \cup_{i=1}^{2}A_i$  there exists $\alpha \in (0,\frac{1}{2})$ such that holds at least  one of the condition Theorem \ref {3.10}.
Thus by Proposition \ref{2.2} and Theorem \ref {3.10} for every $\epsilon>0,$ $F^{\epsilon}_G (T)\neq\emptyset.$
\end{example}

\begin {definition}
Let $\{X_i\}_{i=1}^m$ be nonempty subsets of a $G-metric$  space $X,$  $T:\cup_{i=1}^{m}X_i \rightarrow \cup_{i=1}^{m}X_i$ is a \textbf{G-Mohseni-semi cyclical operator}  if there exists $\alpha\in ]0,\frac{1}{2}[$  such that 
\begin{equation*}
[G(Tx, Ty, Ty)+ G(Ty, Tx, Tx)]\leq\alpha ([G(x, y, y)+  G(y, x, x)]+[G(x, Tx, Tx)+ G(Tx, x, x)]),
\end{equation*}
$\forall x\in X_i,y\in X_{i+1}.$
\end{definition}
\begin{theorem} \label{3.19}
{\it Let $\{X_i\}_{i=1}^m$ be nonempty subsets of a $G-metric$ space $X$ and 
 Suppose $T:\cup_{i=1}^{m}X_i \rightarrow \cup_{i=1}^{m}X_i$ is a G-Mohseni-semi cyclical operator. 
 Then: 
$$\forall \epsilon>0,~ F_\epsilon(T)\neq\emptyset.$$}
\end{theorem} 
{\bf Proof:} Let  $x\in \cup_{i=1}^{m}X_i.$ 
\begin{eqnarray*}
[G(T^{n}x,T^{n+1}x,T^{n+1}x) +G(T^{n+1}x,T^{n}x,T^{n}x)] &=&[G(T(T^{n-1}x),T(T^{n}x),T(T^{n}x)) \\
&+& G(T(T^{n}x),T(T^{n-1}x),T(T^{n-1}x))]\\
&\leq &\alpha [G(T^{n-1}x,T^{n}x,T^{n}x) + G(T^{n}x,T^{n-1}x,T^{n-1}x)]\\
&+& \alpha [G(T^{n-1}x,T^{n}x,T^{n}x) + G(T^{n}x,T^{n-1}x,T^{n-1}x)]\\
&= &2\alpha [G(T^{n-1}x,T^{n}x,T^{n}x) + G(T^{n}x,T^{n-1}x,T^{n-1}x)]\\
& &\vdots\\
&\leq & (2\alpha)^n [G(x, Tx, Tx)+ G(Tx, x, x)]).
\end{eqnarray*}

But $\alpha\in]0,\frac{1}{2}[.$ Therfore
\begin{equation*}
Lim_{n\rightarrow\infty}[G(T^{n}x,T^{n+1}x,T^{n+1}x) +G(T^{n+1}x,T^{n}x,T^{n}x)] =0, ~\forall x\in \cup_{i=1}^{m}X_i.
\end{equation*}
Now by  Lemma \ref{2.5}, it follows that $F^{\epsilon}_G (T)\neq\emptyset, \forall \epsilon>0.$
$~\blacksquare$

\begin{example} \label{4.15} Let $X$ be a subset in $R$ endowed with the usual metric. Suppose $A_1=[0.01,0.8]$ and $A_2=[0.01,\frac{1}{2}].$ Suppose $X_1=[0.01,0.8]$ and $X_2=[0.01,\frac{1}{2}].$ 
 Define the map $T:\cup_{i=1}^{2}A_i\rightarrow \cup_{i=1}^{2}A_i$ as
$Tx=\frac{x}{4}$ for all $x\in\cup_{i=1}^{2}A_i.$
It is easily to be checked that  $T(A_1) \subseteq A_2$  and $T(A_2) \subseteq A_1.$
For any $x \in A_1$ and $y\in A_2$ and Proposition \ref{2.2} we have the chain of inequalities 
\begin{eqnarray*}
[G(Tx, Ty, Ty)+ G(Ty, Tx, Tx)]=d_G(Tx,Ty)&= & |\frac{x}{4}-\frac{y}{4}|\\
&\leq &\frac{1}{3}(|x-y|+|x-\frac{x}{4})\\
&= &\frac{1}{3}(d_G(x,y)+d_G(x,Tx))\\
&=&\frac{1}{3} [G(x, y, y)+  G(y, x, x) + G(x, Tx, Tx)+ G(Tx, x, x)].
\end{eqnarray*}
So $T$ satisfies all the conditions of Theorem \ref{3.19} and thus for every $\epsilon>0,$ $ F^{\epsilon}_G (T)\neq\emptyset.$
\end{example}
\section{ Diameter approximate fixed point for several  operator on $G-$ metric spaces}
In this section,  using Lemma \ref{2.8},  quantitative  results for new  cyclical operators will be formulated  and proved,  and  some results regarding diameter approximate fixed point of such operators on $G-$ metric spaces were given.
\begin{theorem}  \label{5.1}
{\it  Let $\{X_i\}_{i=1}^m$ be nonempty subsets of a metric space $X.$ Suppose that $T:\cup_{i=1}^{m}X_i \rightarrow \cup_{i=1}^{m}X_i$ is a  $G-$Mohseni cyclical operator. Then for every $\epsilon>0,$ $$ \delta (F_{ \epsilon} (T))\leq \frac{2\epsilon(1+\alpha)} {1-2\alpha}.$$}
 \end{theorem}
 
 {\bf Proof:} Let $\epsilon>0.$ Condition i) in Lemma \ref{2.8} is satisfied, as one can see in the proof
of Theorem \ref{3.3} we only verify that condition ii) in  Lemma \ref{2.8} holds.
Let $\theta>0$ and $x,y \in F^{ \epsilon}_G (T)$ and assume that 
Then:
\begin{eqnarray*}
[G(x, y, y)+G(y, x, x)] - [G(Tx, Ty, Ty)+G(Ty, Tx, Tx)]< \theta.
\end{eqnarray*}
Then:
\begin{eqnarray*}
[G(x, y, y)+G(y, x, x)] &\leq& \alpha [G(x, y, y)+G(y, x, x)] + [G(Tx, Ty, Ty)+G(Ty, Tx, Tx)]+\theta.
\end{eqnarray*}
Therefore
As $x,y\in F^{\epsilon} _G(T),$ we know that 
\begin{eqnarray*}
G(x, Tx, Tx)+ G(Tx, x, x)\leq \epsilon,G(y, Ty, Ty)+ G(Ty, y, y)\leq \epsilon.
\end{eqnarray*}
Therfore, $ G(x, y, y)+G(y, x, x)\leq \frac{2\alpha\epsilon+\theta}{1-2\alpha}.$ 
So for every $\theta>0$ there exists $\phi(\theta)=\frac{2\alpha\epsilon+\theta}{1-2\alpha}>0$ such that
\begin{eqnarray*}
[G(x, y, y)+G(y, x, x)] - [G(Tx, Ty, Ty)+G(Ty, Tx, Tx)]< \theta  \Rightarrow G(x, y, y)+G(y, x, x)]\leq\phi(\theta).
\end{eqnarray*} 
 Now by Lemma \ref{2.8}, it follows that $$\delta (F^{ \epsilon}_G (T)) \leq \phi(2\epsilon),\forall \epsilon>0,$$ which means exactly that
\begin{eqnarray*}
\delta (F^{ \epsilon}_G (T)) \leq \frac{2\epsilon(1+\alpha)}{1-2\alpha}.~\blacksquare
\end{eqnarray*}

\begin{example} Let $X$ be a subset in $R$ endowed with the usual metric. Suppose $A_1=[0.01,0.8]$ and $A_2=[0.01,\frac{1}{2}].$ Suppose $X_1=[0.01,0.8]$ and $X_2=[0.01,\frac{1}{2}].$ 
 Define the map $T:\cup_{i=1}^{2}A_i\rightarrow \cup_{i=1}^{2}A_i$ as
$Tx=\frac{x}{4}$ for all $x\in\cup_{i=1}^{2}A_i.$
 \par By example \ref{4.15}  $T:\cup_{i=1}^{m}X_i \rightarrow \cup_{i=1}^{m}X_i$ is a  Mohseni cyclical operator. So $T$ satisfies all the conditions of Theorem \ref{5.1} and thus for every $\epsilon>0,$ $$ \delta (F_{ \epsilon} (T))\leq \frac{2\epsilon(1+\alpha)}
 {1-2\alpha}.$$
 \end{example}
 \begin{theorem}  \label{5.3}
{\it Let $\{X_i\}_{i=1}^m$ be nonempty subsets of a metric space $X.$ Suppose that $T:\cup_{i=1}^{m}X_i \rightarrow \cup_{i=1}^{m}X_i$ is a  G-Mohsenialhosseini cyclical operator.
 Then for every
 $\epsilon>0,$ $$ \delta (F_{ \epsilon} (T))\leq 2\epsilon\frac{1+\eta}{1-\eta},$$
 where $ \eta:=max\{\alpha, \frac{\beta}{1-\beta},\frac{\gamma}{1-\gamma}\},$
 and $\alpha, \beta, \gamma$ as in Definition  \ref{3.9}}.
 \end{theorem}
{\bf Proof:} In the proof of Theorem \ref{3.10}, we have already shown that if $T$ satisfies at least
one of the conditions $(i), (ii), (iii)$ from Definition \ref{3.9}, then
\begin{eqnarray*}
[G(Tx, Ty, Ty)+G(Ty, Tx, Tx)]\leq 2\eta [G(x, Tx, Tx)+ G(Tx, x, x)] + \eta [G(x, y, y)+G(y, x, x)],
\end{eqnarray*}
and
\begin{eqnarray*}
[G(Tx, Ty, Ty)+G(Ty, Tx, Tx)]\leq 2\eta [G(x, Ty, Ty)+ G(Ty, y, y)] + \eta  [G(x, y, y)+G(y, x, x)],
\end{eqnarray*}
hold.
Let $\epsilon>0.$  We will only verify that condition (ii) in Lemma \ref{2.8} is satisfied,
as (i) holds, see the Proof of  Theorem \ref{3.10}.\\
Let $\theta>0$ and $x,y \in F^{\epsilon}_G (T),$ and assume that $[G(x, y, y)+G(y, x, x)]-[G(Tx, Ty, Ty)+G(Ty, Tx, Tx)]\leq \theta.$ Then
\begin{eqnarray*}
[G(x, y, y)+G(y, x, x)]& \leq & [G(Tx, Ty, Ty)+G(Ty, Tx, Tx)]+ \theta\Rightarrow
\end{eqnarray*}
\begin{eqnarray*}
[G(x, y, y)+G(y, x, x)] &\leq & 2\eta [G(x, Tx, Tx)+ G(Tx, x, x)] +\eta [G(x, y, y)+G(y, x, x)] +\theta\Rightarrow
\end{eqnarray*}
\begin{eqnarray*}
(1-\eta)[G(x, y, y)+G(y, x, x)]\leq 2\eta\epsilon+\theta
\end{eqnarray*}
\begin{eqnarray*}
[G(x, y, y)+G(y, x, x)]\leq \frac{2\eta \epsilon+\theta}{1-\eta}.
\end{eqnarray*}
So for every $\theta>0$ there exists $\phi(\theta)=\frac{2\eta \epsilon+\theta}{1-\eta}>0$ such that 
$$[G(x, y, y)+G(y, x, x)]-[G(Tx, Ty, Ty)+G(Ty, Tx, Tx)]\leq \theta \Rightarrow [G(x, y, y)+G(y, x, x)]\leq\phi(\theta).$$ Now by Lemma \ref{2.8}, it follows that $$\delta (F_{ \epsilon} (T)) \leq \phi(2\epsilon),\forall \epsilon>0,$$ which means exactly that
\begin{eqnarray*}
\delta (F_{ \epsilon} (T)) \leq 2\epsilon\frac{1+\eta}{1-\eta},~\forall \epsilon>0.~\blacksquare
\end{eqnarray*}

\begin{example} Let $X=[0,\infty)$ and let $d$ be usual metric on $X.$ Suppose $A_1=[0.1,2]$ and $A_2=[0.1,1].$ Fix $\beta\in (0,1)$ and define  $T:\cup_{i=1}^{2}A_i\rightarrow \cup_{i=1}^{2}A_i$ as 
\[Tx=\left\{\begin {array}{cl}
0\quad \quad  x\in [0,1-\beta)\\
\frac{x}{4}\qquad   x\in [1-\beta,1)\\
\frac{1-\beta}{4}\qquad \quad x\in [1,2]
 \end{array}\right.\]
 \par By example \ref{4.12}  $T:\cup_{i=1}^{m}X_i \rightarrow \cup_{i=1}^{m}X_i$ is a  G-Mohsenialhosseini cyclical operator. So $T$ satisfies all the conditions of Theorem \ref{5.3} and thus for every $\epsilon>0,$
  It is easy to  check that $T(A_1) \subseteq A_2$  and $T(A_2) \subseteq A_1.$
For any $x,y \in \cup_{i=1}^{2}A_i$  there exists $\alpha \in (0,\frac{1}{2})$ such that holds at least  one of the condition Theorem \ref {3.10}.
Thus by Theorem \ref {3.10} for every $\epsilon>0,$ $F^{\epsilon}_G (T)\neq\emptyset.$
\end{example}
\begin{theorem}  \label{4.5} 
{\it Let $\{X_i\}_{i=1}^m$ be nonempty subsets of a metric space $X.$ Suppose that $T:\cup_{i=1}^{m}X_i \rightarrow \cup_{i=1}^{m}X_i$ is a  G-Mohseni-semi cyclical operator. Then for every $\epsilon>0,$ $$ \delta (F^{\epsilon}_G (T))\leq \epsilon\frac{2+\alpha}{1-\alpha}.$$}
 \end{theorem}
{\bf Proof:} Let $\epsilon>0.$  We will only verify that condition 2) in Lemma \ref{2.8} is satisfied.
Let $\theta>0$ and $x,y \in F^{\epsilon}_G (T),$ and assume that $([G(x, y, y)+G(y, x, x)]-[G(Tx, Ty, Ty)+G(Ty, Tx, Tx)])\leq \theta.$ Then
\begin{eqnarray*}
[G(x, y, y)+G(y, x, x)]\leq  [G(Tx, Ty, Ty)+G(Ty, Tx, Tx)]+\theta\Rightarrow
\end{eqnarray*}
\begin{eqnarray*}
[G(x, y, y)+G(y, x, x)]\leq \alpha [[G(x, y, y)+G(y, x, x)]+[G(x, Tx, Tx)+G(Tx, x, x)] ]+\theta\Rightarrow
\end{eqnarray*}
$$(1-\alpha)[G(x, y, y)+G(y, x, x)]\leq\alpha [G(x, Tx, Tx)+G(Tx, x, x)]+\theta$$
\begin{eqnarray*}
[G(x, y, y)+G(y, x, x)]\leq \frac{\alpha \epsilon+\theta}{1-\alpha}
\end{eqnarray*}
So for every $\theta>0$ there exists $\phi(\theta)=\frac{\alpha \epsilon+\theta}{1-\alpha}>0$ such that 
\begin{eqnarray*}
[G(x, y, y)+G(y, x, x)]-[G(Tx, Ty, Ty)+G(Ty, Tx, Tx)]\leq \theta \Rightarrow [G(x, y, y)+G(y, x, x)]\leq\phi(\theta)
\end{eqnarray*}
Now by Lemma \ref{2.8}, it follows that $$\delta (F_{ \epsilon} (T)) \leq \phi(2\epsilon),\forall \epsilon>0,$$ which means exactly that
\begin{eqnarray*}
\delta (F_{ \epsilon} (T)) \leq \epsilon\frac{2+\alpha}{1-\alpha},~\forall \epsilon>0.~\blacksquare
\end{eqnarray*}

\begin{remark}  
{\it Examples  \ref{3.8} and \ref{4.12} holds in Theorem \ref{4.5}.}
\end{remark}


\end{document}